\documentclass{autart}

\usepackage[dvipsnames]{xcolor}
\usepackage[colorlinks=true,citecolor=MidnightBlue,draft]{hyperref}

\usepackage[british]{babel} 
\usepackage{natbib}

\usepackage{amsfonts,amsmath,amssymb,mathtools}
\newtheorem{theorem}{Theorem}

\newtheorem{lemma}{Lemma}

\newtheorem{assumption}{Assumption}
\newtheorem{proposition}{Proposition}
\theoremstyle{definition}

\newtheorem{remark}{Remark}
\newenvironment{proof}{}{}
\newenvironment{proof*}{}{}

\usepackage{graphicx,subcaption,caption}
\usepackage{tikz}

\usepackage{booktabs}
\usepackage{float}
\newfloat{equationfloat}{b}{eqn}
\usepackage{siunitx}
\usepackage{bbold}
\usepackage{arydshln} 
\ADLdrawingmode{3}

\usepackage{flushend} 

\usepackage{xspace}

\newcommand\laplacian{\ensuremath{\mathcal{L}}\xspace}
\newcommand\incidence{\ensuremath{\mathcal{B}}\xspace}
\newcommand\cardin{\ensuremath{n}\xspace}
\newcommand\cardinedge{\ensuremath{m}\xspace}
\newcommand\neighbor{\ensuremath{\mathcal{N}}\xspace}
\newcommand\average{\ensuremath{\Pi}\xspace}
\newcommand\rotor{\ensuremath{\theta}\xspace}
\newcommand\redge{\ensuremath{\eta}\xspace}
\newcommand\aedge{\ensuremath{\delta}\xspace}
\newcommand\rotorset{\ensuremath{\Theta}\xspace}
\newcommand\freq{\ensuremath{\omega}\xspace}
\newcommand\ctrl{\ensuremath{\xi}\xspace}
\newcommand\inp{\ensuremath{u}\xspace}
\newcommand\demand{\ensuremath{P}\xspace}
\newcommand\vedge{\ensuremath{\Gamma}\xspace}
\newcommand\voltage{\ensuremath{V}\xspace}
\newcommand\suscept{\ensuremath{B}\xspace}
\newcommand\damp{\ensuremath{D}\xspace}
\newcommand\inert{\ensuremath{M}\xspace}
\newcommand\balance{\ensuremath{Q}\xspace}
\newcommand\pot{\ensuremath{U}\xspace}
\newcommand\laplcomm{\ensuremath{\laplacian_\ctrl}\xspace}

\newcommand\lyap{\ensuremath{W}\xspace}
\newcommand\statevec{\ensuremath{x}\xspace}
\newcommand\lyapvec{\ensuremath{\chi}\xspace}
\newcommand\lyapmat{\ensuremath{K}\xspace}

\newcommand\diff[1]{\ensuremath{#1 - \bar{#1}}\xspace}
\newcommand\potdiff{\ensuremath{\nabla\pot(\aedge) - \nabla\pot(\bar\aedge)}\xspace}
\newcommand\opt[1]{\ensuremath{{#1}^{\rm opt}}\xspace}
\newcommand\cost{\ensuremath{C}\xspace}
\newcommand\ep[1]{\ensuremath{\epsilon_{#1}}\xspace}
\newcommand\al[1]{\ensuremath{\alpha_{#1}}\xspace}
\newcommand\be[1]{\ensuremath{\beta_{#1}}\xspace}
\newcommand\filter{\ensuremath{\mathcal{U}}\xspace}
\newcommand\filterall{\ensuremath{\mathcal{V}}\xspace}
\newcommand\ctrld{\ensuremath{z}\xspace}
\newcommand\ctrla{\ensuremath{\zeta}\xspace}
\newcommand\ctrlm{\ensuremath{\mu}\xspace}
\newcommand\bound{\ensuremath{c}\xspace}
\newcommand\lowerbd{\ensuremath{\underline{\bound}}\xspace}
\newcommand\upperbd{\ensuremath{\overline{\bound}}\xspace}
\newcommand\derivbd{\ensuremath{{\bound^\prime}}\xspace}
\newcommand\hesslev{\ensuremath{\underline\nu}\xspace}

\newcommand\levelset{\ensuremath{\Omega}\xspace}
\newcommand\cardindos{\ensuremath{k}\xspace}
\newcommand\dositv{\ensuremath{H}\xspace}
\newcommand\dosallitv{\ensuremath{\Xi}\xspace}
\newcommand\doslen{\ensuremath{|\dosallitv(t)|}\xspace}
\newcommand\dosstart{\ensuremath{h}\xspace}
\newcommand\dosduration{\ensuremath{\dosboundrate}\xspace}
\newcommand\dosboundstart{\ensuremath{\kappa}\xspace}
\newcommand\dosboundrate{\ensuremath{\tau}\xspace}
\newcommand\R{\ensuremath{\mathbb{R}}\xspace}
\newcommand\Z{\ensuremath{\mathbb{Z}}\xspace}

\newcommand\0{\ensuremath{\mathbb{0}}\xspace}
\newcommand\1{\ensuremath{\mathbb{1}}\xspace}
\newcommand\ones{\ensuremath{\1\1\T}\xspace}
\newcommand\lmax[1]{\ensuremath{\lambda_{\rm max}(#1)}\xspace}
\newcommand\lmin[1]{\ensuremath{\lambda_{\rm min}(#1)}\xspace}
\newcommand\T{^\top}
\newcommand\I{^{-1}}
\newcommand\gen{_\gens}
\newcommand\load{_\loads}
\newcommand\gens{\ensuremath{\mathbf{G}}\xspace}
\newcommand\loads{\ensuremath{\mathbf{L}}\xspace}

\newcommand\dd[2]{\frac{\partial #1}{\partial #2}}
\newcommand\lead{\multicolumn{2}{c}{}}
\DeclareMathOperator{\diag}{\mathrm{diag}}
\DeclareMathOperator{\range}{\mathrm{Im}}

\DeclareMathOperator{\bdiag}{\mathrm{block\,diag}}
\DeclareMathOperator{\col}{\mathrm{col}}
\DeclareMathOperator{\symm}{sp}
\renewcommand{\exp}[1]{\ensuremath{\mathrm{e}^{#1}}}
\renewenvironment{proof}{%
\par\noindent{\itshape Proof. }%
\setcounter{custprftag}{0}%
\newcommand\custprfendprf{\ifnum\value{custprftag}=0{\hfill$\square$\par}\else{\vspace{-1em}}\fi}%
\newcommand\QEDhere{\tag*{$\square$}\stepcounter{custprftag}}%
}{\custprfendprf}
\renewenvironment{proof*}[1]{%
\par\noindent{\itshape #1. }%
\setcounter{custprftag}{0}%
\newcommand\custprfendprf{\ifnum\value{custprftag}=0{\hfill$\square$\par}\else{\vspace{-1em}}\fi}%
\newcommand\QEDhere{\tag*{$\square$}\stepcounter{custprftag}}%
}{\custprfendprf}
\newcounter{custprftag}

\newcommand{\futureconsideration}[1]{} 

\begin{document}

\begin{frontmatter}
\title{Exponential convergence under\\distributed averaging integral frequency control}
\thanks{The material in this paper was partially presented at the 2017 IFAC World Congress, Toulouse, France.}
\date  {}
\author[g]{Erik Weitenberg}
\author[g]{Claudio De Persis}
\author[c]{Nima Monshizadeh}

\address[g]{Engineering and Technology Institute Groningen and\\ Jan Willems Center for Systems and Control\\ University of Groningen, 9747 AG Groningen, The Netherlands\\ (e-mail: \{e.r.a.weitenberg,c.de.persis\}@rug.nl)}
\address[c]{Department of Engineering, University of Cambridge\\ Trumpington Street, Cambridge, CB2 1PZ, United Kingdom\\ (e-mail: n.monshizadeh@eng.cam.ac.uk)}
\begin{abstract}
We investigate the performance and robustness of  distributed averaging integral controllers used in the optimal frequency regulation of power networks.
We construct a strict Lyapunov function that allows us to quantify the exponential convergence rate of the closed-loop system. 
As an application, we study the stability of the system in the presence of disruptions to the controllers' communication network, and investigate  how the convergence rate is affected by these disruptions.
\end{abstract}


\end{frontmatter}



\section{Introduction} \label{sec:intro} 

Modern power grids can be regarded as a large network of control areas, each producing and consuming power and transferring it to adjacent areas.
The frequency of the AC signal is tightly regulated around its nominal value of e.g.\ \SI{50}{Hz} to guarantee reliable operation of this network.
Traditionally, this is achieved by means of proportional (`droop') control and PI control.
In this setup, each area compensates for its local fluctuations in load, and adjusts its production to provide previously scheduled power flows to the adjacent areas.
As a result, estimates of the load in each area are required in advance to achieve economical efficiency.

Recently, renewable energy sources such as wind turbines have been introduced in significant numbers.
Since these sources do not usually provide a predictable amount of power, the net load on the individual control areas will change more rapidly and by larger amounts.
More substantial fluctuations are expected to occur in microgrids, which are energy systems that can operate independently of the main grid.
The resulting need for more advanced control strategies for future power networks has led to the design of distributed controllers equipped with a real-time communication network \citep{dorfler14breaking,shafiee14distributed,mojica14dynamic,buerger13dynamic,trip16internal}.

The addition of a communication network raises a reliability and security problem, as communication packets can be lost and digital communi\-cation networks may fall victim to failures and malicious attacks.
A common disruption is the so-called Denial of Service, or DoS \citep{byres04myths}, which can be understood as a partial or total interruption of communications.
It is therefore of interest to characterize the performance degradation of the aforementioned networks of distributed controllers under loss of information, possibly due to a DoS event.

\subsection{Literature review} 

The current research on frequency regulation in power networks is reviewed in \cite{ibraheem05recent}. Since this field of research receives considerable amounts of attention, we will summarize a subset of results that are close to our interest.

Frequency stability and control in power networks is a well-established field of research which has lead to important results for a variety of models \citep[see e.g.][]{bergen81structure, tsolas85structure}.
More recently, distributed control methods have been proposed to guarantee not only frequency regulation but also economic optimality.
In a microgrid context, distributed averaging integral control is well-studied \citep{simpson-porco13synchronization, buerger13dynamic,dorfler14breaking,trip16internal,andreasson17coherence}.
In the context of power networks, distributed internal-model-based optimal controllers have also been studied \citep{buerger13dynamic, trip16internal}.
As a complementary approach to distributed integral or internal-model controllers, primal-dual gradient controllers \citep{li14connecting,zhang13real,stegink16unifying} are able to handle general convex objective functions as well as constraints, but in turn require much information about the power network parameters.

The robustness of power networks under various controllers has been investigated in the works above to varying degree.
In this light, it is useful to consider strictly decreasing energy functions \citep{malisoff09constructions}.
\citet{zhao15distributed} make a first attempt to arrive at one, and their effort is expanded upon by \citet{schiffer17robustness} in the context of time-delayed communication.
Bearing this in mind, we propose a construction of a new strict Lyapunov function for the purpose of explicitly quantifying the 
\emph{exponential} convergence of power networks under distributed averaging integral control and then study the performance of this control in the presence of communication disruptions.

As an application of robustness measures, we will investigate the effect of Denial of Service.
It, and related phenomena, have been studied as well.
See e.g.\ \cite{byres04myths} for an introduction to the subject.
It can be modeled as a stochastic process \citep{befekadu15risk}, a resource-constrained process \citep{gupta10optimal}, or using only constraints on the proportion of time it is active \citep{persis15input,persis14resilient-ifac}.
Correspondingly, the investigations of systems under DoS events vary, with focus being on planning transmissions outside the disruption intervals \citep{shishehforoush13triggering},
  limiting the maximum ratio of time during which DoS is active \citep{persis15input},
  or guaranteeing stability regardless and quantifying convergence behavior \citep{persis15input,persis14resilient-ifac}.
The latter approach offers interesting perspectives, since the specific characterization of the period of time during which communication is not permitted adopted in \cite{persis14resilient-ifac} allows for great flexibility and can conveniently model both genuine loss of communication or packet drops due to malicious behaviour.
Furthermore, the analysis of \cite{persis15input,persis14resilient-ifac} is based on Lyapunov functions, can handle
distributed systems \citep{senejohnny15selftriggered,senejohnny16jamming}, and therefore is well suited for the class of nonlinear networked models describing power networks.

\subsection{Main contribution} 

The contribution of this paper is primarily theoretical: existing approaches to the  problem of optimal frequency control have mostly relied on non-strictly decreasing energy -- or Lyapunov functions, using LaSalle's invariance principle and related results to guarantee convergence to an invariant mani\-fold on which the Lyapunov function's derivative vanishes (see \citealp{schiffer17robustness,turitsyn17framework} for exceptions).
Since this does not lead to strong results on  convergence, we design a strictly decreasing Lyapunov function that does
prove exponential convergence to the optimal synchronous solution.

As an illustration, the final  part of the paper makes use of the developed Lyapunov function to show exponential convergence to the optimal solution in spite of possible communication interruptions, and directly relate the speed of convergence to the physical parameters of the system and the availability of the communication network.
As a result, the resilience of the aforementioned economically optimal control strategies to DoS events is quantified explicitly.


The remainder of this paper is organized as follows.
In Section~\ref{sec:setting}, we outline our model for the power network, goals for its control, and existing control strategies we will use.
Then, in Section~\ref{sec:lyapunov}, we derive a strictly decreasing Lyapunov function and show exponential convergence to the optimal solution.
In Section~\ref{sec:dos}, we introduce a model for communication disruptions, and use our Lyapunov function to 
study the robustness of distributed controllers to these disruptions. 
In Section~\ref{sec:simulations}, we illustrate the main result using numerical simulations of an academic model of a power network.
Finally, Section~\ref{sec:conclusions} presents conclusions.

\subsection{Notation} \label{ssec:notation} 
Given a system state $x = x(t)$, we use the notation $\dot x$ to mean the time derivative $\dd xt$. Likewise, a function $f: \R^n \to \R$ of such a state, such as a Lyapunov function, has time derivative $\dot f := (\nabla_x f(x))\T \dot x$. We denote its Hessian by $\nabla^2 f$.
When used with vector arguments, $\sin$ and $\cos$ are defined element-wise.
The symbols \0 and \1 denote vectors and matrices filled with 0 and 1 respectively; if there is ambiguity about their size, the dimensions are given as a subscript.
Finally, $\symm(A) := \frac12 (A + A\T)$ is used to denote the symmetric part of a square matrix $A$.

\section{Setting} \label{sec:setting} 

We consider a power grid, represented here by a set of \cardin buses.
The network of power lines between the buses is represented by a connected graph with \cardin nodes and \cardinedge arbitrarily oriented edges and with $\pm 1$-valued incidence matrix \incidence. The orientation is necessary for analytical purposes but otherwise meaningless; the physical network is undirected.

We will use a structure-preserving model for the power network.
We consider two types of nodes.
Some nodes in the network are connected to synchronous generators or inverters with filtered power measurements; these we call generators.
The others, which we will refer to as loads, are frequency-responsive loads or inverters with instantaneous power measurements and primary droop control.
In this work, we disregard the additional possibility of `passive' nodes that do not contribute to frequency control at all.
Accordingly, we define the sets \gens and \loads of generator and load nodes with cardinality $\cardin\gen$ and $\cardin\load$ respectively, such that $\cardin\gen + \cardin\load = \cardin$.

The dynamics at each bus is considered in a reference frame that rotates with a certain nominal frequency, i.e.\ \SI{50}{Hz}. 
The dynamics can be expressed in the following form, also known as the swing equations \citep{kundur94power}. At generator node $i \in \gens$, 
\begin{subequations}\label{e:swingpernode}
\begin{align}
\dot\rotor_i &= \freq_i  \label{e:swingpernode1}\\
\inert_i \dot\freq_i
             &= -\damp_i \freq_i - \sum_{\mathclap{j\in\neighbor_i}} \gamma_{ij} \sin(\rotor_i - \rotor_j)
                + \inp_i - \demand_i, \label{e:swingpernode2}\\
\intertext{whereas at load node $i \in \loads$,}
0            &= -\damp_i \freq_i - \sum_{\mathclap{j\in\neighbor_i}} \gamma_{ij} \sin(\rotor_i - \rotor_j)
                + \inp_i - \demand_i, \label{e:swingpernode3}
\end{align}
\end{subequations}
Here, $\gamma_{ij} = \suscept_{ij} \voltage_i \voltage_j$ for each edge connecting buses $i$ and $j$.
We summarize the symbols used in Table \ref{t:symbols}. In this paper, we assume that the voltages at the buses are constant and the lines are lossless.

\begin{remark}[Microgrid model]\label{rem:model}
The model \eqref{e:swingpernode} represents a fairly general model of power networks and microgrids.
The system \eqref{e:swingpernode} is known as the structure-preserving model of power network  \citep{bergen81structure,chiang95direct}, where the load and generator buses are differentiated, and the net active power drawn by a load is an affine function of the frequency at that bus.  
Moreover, the dynamics at the nodes \eqref{e:swingpernode1}-\eqref{e:swingpernode2} can also be associated with droop-controlled inverters with power measurement filters in microgrids \citep{schiffer14conditions}.
Finally, we emphasize that the presence of controllable demand $\inp_i$ at the load buses is optional, and the Lyapunov analysis can be carried out for the same network without controllable demands.
\end{remark}

Inspired by the {\em center-of-inertia coordinates} in classic power system multi-machine stability studies \citep{sauer98power}, we define the average of the phase angles of the inverters as the reference, i.e., $\aedge = \average \rotor$, with $\average := I - \frac1\cardin \ones.$
Note that for any incidence matrix, $\incidence\T \average = \incidence$, since $\1\in\ker(\incidence\T)$.
For ease of computation, we will write the dynamics \eqref{e:swingpernode} in the vector form as follows:
\begin{align} \label{e:swing}
\begin{split}
\dot\aedge       &= \average \freq \\
\inert \dot\freq &= -\damp \freq - \incidence \vedge \sin(\incidence\T \aedge) + \inp - \demand.
\end{split}
\end{align}
Whenever a variable or parameter is used without subscript, it refers to the concatenated version; e.g. $\freq := \col(\freq\gen, \freq\load)$, $\vedge = \diag(\gamma_1, \ldots,\gamma_m)$, $\damp := \bdiag(\damp\gen, \allowbreak \damp\load)$ and $\inert = \bdiag(\inert\gen, \0_{\cardin\load \times \cardin\load})$.

\begin{table} \centering
\captionsetup{width=\linewidth} 
\caption{Symbols and parameters used in the system model} \label{t:symbols}
\footnotesize
\begin{tabular}{l@{ $\in$ }ll}
\lead       & \bf State variables \\
\midrule
\rotor      & $\R^\cardin$ & Voltage phase angles at the edges \\
\freq       & $\R^\cardin$ & Frequency deviations at the nodes \\
\ctrl       & $\R^\cardin$ & Controller states at the nodes \\
[1ex] \lead & \bf Input \\
\midrule
\inp        & $\R^\cardin$ & Controllable generation ($+$) or demand ($-$) \\
\demand     & $\R^\cardin$ & Constant demand ($+$) or generation ($-$) \\
[1ex] \lead & \bf Network \\
\midrule
\incidence  & $\Z^{\cardin \times \cardinedge}$ & Incidence matrix \\
\laplacian  & $\Z^{\cardin \times \cardin}$     & Laplacian matrix \\
\multicolumn{2}{l}{$\neighbor_i$}               & Set of nodes neighboring node $i$ \\
[1ex] \lead & \bf Physical parameters \\
\midrule
\inert      & $\R_+^{\cardin \times \cardin}$
            & Moments of inertia as diagonal matrix \\
\damp       & $\R_+^{\cardin \times \cardin}$
            & Damping constants as diagonal matrix \\
\voltage    & $\R^\cardin$
            & Vector of voltages at the buses \\
\suscept    & $\R^{\cardinedge \times \cardinedge}$
            & Matrix of susceptances of the power lines \\
\balance    & $\R_+^{\cardin \times \cardin}$
            & Diagonal matrix of generation costs \\
\end{tabular}
\end{table}

\subsection{Control goal} 
A primary goal in control of power networks is to regulate the frequency deviation to zero.
Let $u=\bar u$ with $\bar u$ being a constant vector. Then, for an equilibrium $(\bar\delta, \bar\omega)$ of \eqref{e:swing} with $\bar\omega=0$, we have
\begin{equation} \label{e:swingsteadystate}
\0 = - \incidence \vedge \sin(\incidence\T \bar\aedge) + \bar\inp - \demand.
\end{equation}
Under the assumption, which we will formalize later, that a solution to \eqref{e:swingsteadystate} exists, there are an infinite number of choices for the input $\bar\inp$ to satisfy \eqref{e:swingsteadystate} given a constant demand \demand.
This freedom can be exploited to design an input $\bar\inp$ which is optimal according to some suitable objective function.

As a matter of fact, in modern power systems, generators do not always have the same capacity.
For this reason, a controller structure that allows the more powerful, cheaper generators to do most of the work are more attractive.
The controllers used in the following sections make use of the concept of distributed optimal  power dispatch which has been investigated in e.g.\ \cite{dorfler14breaking, trip16internal} and references therein.
In this framework, we consider the cost to be dependent only on the amount of power produced, as transmission and other costs are relatively small.
Each generator input $\bar\inp_i, i = 1, \ldots, \cardin$, is assigned a convex cost function $\cost_i(\bar\inp_i)$.
We can then define an overall convex cost function $\cost(\bar\inp) = \sum_{i=1}^n \cost_i(\bar\inp_i)$ and cast the following static optimization problem:
\begin{align}
\begin{split} \label{e:costmin}
\min_{\bar\inp} &\; \cost(\bar\inp) \\
\text{subject to} &\; \1\T(\bar\inp - \demand) = 0.
\end{split}
\end{align}
An optimal steady state solution to \eqref{e:swing} is therefore defined as the one that minimizes the costs of power generation while balancing power supply and demand.
Remarkably, it can be shown that, under a suitable assumption, the optimization problem \eqref{e:costmin} is \emph{equivalent} to the problem
\begin{align}
\begin{split} \label{e:costmin:eq}
\min_{\bar\inp, \bar\aedge} &\; \cost(\bar\inp) \\
\text{subject to} &\; \0 = - \incidence \vedge \sin(\incidence\T \bar\aedge) + \bar\inp - \demand,
\end{split}
\end{align}
which highlights the relevance of \eqref{e:costmin} to the  cost minimzation problem subject to the steady state constraint \eqref{e:swingsteadystate} \citep[Lemma~4]{trip16internal}.

The problem of economic dispatch was addressed by the distributed controllers introduced concurrently and independently in number of papers, which we cover next.
The main objective of this work is to explicitly characterize the performance of these controllers, that is, the speed at which the system converges to its optimal solution.
Then, their robustness against communication disruptions, to be defined precisely in Subsection~\ref{ss:dosintro}, is made explicit as well.

\subsection{Economically optimal controller} 
In this subsection, we briefly recall the control strategy detailed in e.g.\ \cite{dorfler14breaking,dorfler13synchronization,monshizadeh16agreeing,trip16internal,persis15bregman,simpson-porco13synchronization}.
In the following material, we will assume cost function \cost to be quadratic, i.e., $\cost_i(\bar \inp_i) = \frac12 q_i \bar\inp_i^2, \; q_i > 0$.
Restricting it to this form allows to avoid load and/or power flow measurements.
Writing $\cost(\inp) = \frac12 \inp\T \balance \inp$, with $Q=\diag(q_i)$,
we introduce the Lagrangian function $L(\inp, \lambda) = \cost(\inp) + \lambda \1\T(\inp - \demand)$, where $\lambda \in \R$ denotes the Lagrange multiplier.
Noting that $L$ is strictly convex in \inp and concave in $\lambda$, there is a saddle point solution $(\bar\inp, \bar\lambda)$ to $\max_\lambda \min_\inp L(\inp, \lambda)$ satisfying
\begin{align*}
\nabla \cost(\bar\inp) + \1 \bar\lambda &= \0 \\
\1\T(\bar\inp - \demand) &= 0,
\end{align*}
which is obtained as \citep[Lemma 3]{trip16internal}
\begin{equation} \label{e:uoptimal}
\opt{\bar\inp} = \balance\I \frac{\ones \demand}{\1\T \balance\I \1}.
\end{equation}
Note that at the optimal point \eqref{e:uoptimal}, the power generated at each node $i$ is proportional to the inverse of its marginal cost $\balance_i$.

Now, returning to equality \eqref{e:swingsteadystate} and setting $u=\opt{\bar\inp}$ yields
\begin{equation} \label{e:feas}
\0 = - \incidence \vedge \sin(\incidence\T \bar\aedge) + \balance\I \frac{\ones \demand}{\1\T \balance\I \1} - \demand,
\end{equation}
which together with $\bar\omega=0$ identify an equilibrium of \eqref{e:swing} with zero frequency deviation and optimal power dispatch.
Due to the presence of the sinusoids, the first term in the right-hand-side of the equality above is bounded, and thus an arbitrary mismatch between the optimal generation $\opt{\bar\inp}$ and demand $P$ cannot be tolerated.  Therefore, we impose the following feasibility assumption to guarantee the existence of an equilibrium with optimal properties:
\begin{assumption}[Feasability] \label{asm:existence}
There exists a vector $\bar\delta\in \range \Pi$ such that \eqref{e:feas} is satisfied, and $\incidence\T \bar\aedge$ is in the \emph{interior} of $\rotorset := [\rho - \frac\pi2, \allowbreak \frac\pi2 - \rho]^\cardinedge$, for some $\rho > 0$.
\end{assumption}

\begin{remark}[Security constraint]
The extra condition on $\bar\aedge$ is standard in power grid stability investigations and is usually called the security constraint \citep{dorfler14breaking}.
We modify it slightly by making explicit the distance of $(\incidence\T \bar\aedge)_i$ from $\pm \pi/2$.
This will be necessary later to show boundedness of the trajectories of \eqref{e:swing}, and to derive explicit expressions for its rate of decay.
We require that the equilibrium is in the interior of this set, so a bounded open set around it will always exist in which to prove exponential convergence of trajectories.
\end{remark}

The following lemma shows that under Assumption \ref{asm:existence}, the solution $\bar\delta$ to \eqref{e:feas} is unique.
\begin{lemma}[Uniqueness of the equilibrium]
Let Assumption \ref{asm:existence} hold for some vectors $\bar\delta$ and $\bar\delta'$. 
Then, $\bar\delta=\bar\delta'$.
\end{lemma}
\begin{proof}
Note that, because $\aedge \in \range(\average)$, there exists a vector $\bar\rotor$ such that $\bar\aedge = \average \bar\rotor$, and likewise, there exists a vector $\bar\rotor'$ with $\bar\aedge' = \average \bar\rotor'$.
Then, applying \eqref{e:swingsteadystate} to both $\bar\aedge$ and $\bar\aedge'$ yields
\[ \incidence \vedge \sin(\incidence\T \bar\aedge) 
 - \incidence \vedge \sin(\incidence\T \bar\aedge') = \0. \]
Multiplying on the left by $\bar\rotor - \bar\rotor'$, we get
\[ (\bar\rotor-\bar\rotor')\T \incidence \vedge (\sin(\incidence\T \bar\aedge) - \sin(\incidence\T \bar\aedge')) = 0, \]
or equivalently,
\[ (\bar\aedge-\bar\aedge')\T \incidence \vedge (\sin(\incidence\T \bar\aedge) - \sin(\incidence\T \bar\aedge')) = 0. \]
Note that \vedge is a diagonal matrix.
Hence, we can expand this statement, writing $\incidence\T_i$ for the $i^{\rm th}$ row of $\incidence\T$, as
\[ \sum_{i=1}^\cardin (\incidence\T_i (\bar\aedge - \bar\aedge')) \gamma_i (\sin(\incidence\T_i \bar\aedge) - \sin(\incidence\T_i \bar\aedge')) = 0. \]
Since $\sin$ is strictly monotonous on \rotorset and since $\incidence\T \aedge \in \rotorset$, and since $\gamma_i > 0$ for all $i$, each of the terms of this sum is nonnegative.
As the sum vanishes, each term must be 0; we must conclude that $\incidence\T \bar\aedge = \incidence\T \bar\aedge'$, hence $\bar\aedge - \bar\aedge' \in \ker(\incidence\T) = \range(\1)$.
By definition, $\bar\aedge - \bar\aedge' \in \range(\average) \perp \range(\1)$.
Therefore, $\bar\aedge = \bar\aedge'$.
\end{proof}

\futureconsideration{if space is an issue, we can omit the proof above.}

We now introduce the distributed control algorithm \citep{simpson-porco13synchronization,trip16internal,monshizadeh16agreeing,dorfler13synchronization,zhao15distributed}.
At each node, a controller actuates the local energy production $\inp_i$.
Economic optimality is achieved by fitting the controllers with an undirected, connected, delay-free communication network, represented by a graph with Laplacian matrix \laplcomm.
The dynamics of the controllers at the nodes are then given by
\begin{align} \label{e:controller}
\dot \ctrl_i &= - \sum_{\mathclap{j\in\neighbor_{\mathrm{comm}, i}}}
                  (\balance_i \ctrl_i - \balance_j \ctrl_j)
                - \balance_i\I \freq_i
   &  \inp_i &= \ctrl_i, & i \in \gens \cup \loads \nonumber \\
\intertext{defining $\neighbor_{\mathrm{comm}, i}$ as the set of neighbors of node $i$ in the communication network. In the vector form the expression becomes}
\dot \ctrl &= - \laplcomm \balance \ctrl - \balance\I \freq &  \inp &= \ctrl.
\end{align}

\begin{proposition}
Under Assumption \ref{asm:existence}, the solutions to the system \eqref{e:swing} in closed loop with the controllers at the nodes \eqref{e:controller} locally%
\footnote{The term locally refers to the fact that solutions are initialized in a suitable neighborhood of $(\bar\aedge, \0, \bar\ctrl)$.}%
converge to the point $(\aedge, \freq, \ctrl) = (\bar\aedge, \0, \bar\ctrl := \opt{\bar\inp})$.%
\footnote{The proof follows immediately from \citet[Thm. 4]{monshizadeh15output}}
\end{proposition}

The implication of this proposition is that the distributed controllers \eqref{e:controller} are able to regulate the frequency to its nominal value and achieve economically optimal generation of power without measuring the total fixed demand and generation, \demand.

\section{Strictly decreasing Lyapunov function} \label{sec:lyapunov} 

To arrive at an exponential bound on the speed of convergence, we first construct a strictly decreasing Lyapunov function.
We then derive an exponentially decreasing upper bound for the Lyapunov function value, and discuss its implications.

\subsection{Strict Lyapunov function} 

The analysis below makes heavy use of an incremental model of the original system \eqref{e:swing}, \eqref{e:controller}, with respect to the equilibrium $(\bar\aedge, \0, \bar\ctrl)$, $\bar\ctrl = \opt{\bar\inp}$.
This gives rise to the following dynamics:
\begin{align}
\begin{split}\label{e:incrementalswing}
\dot\aedge   &= \average \freq \\
\inert\gen \dot\freq\gen
             &= -\damp\gen \freq\gen - (\potdiff)\gen + \diff{\ctrl\gen} \\
\0           &= -\damp\load \freq\load - (\potdiff)\load + \diff{\ctrl\load} \\
\dot\ctrl    &= -\laplcomm \balance (\diff\ctrl) - \balance\I \freq
\end{split}
\end{align}
where $\pot(\aedge) = -\1\T \vedge \cos(\incidence\T \aedge)$ is the so-called potential function whose gradient satisfies $\nabla\pot(\aedge) = \incidence \vedge \sin(\incidence\T \aedge)$.
We denote the subvector $\incidence\gen \vedge \sin(\incidence\T \aedge)$ by the shorthand $\nabla\pot(\aedge)\gen$, and likewise for $\nabla\pot(\aedge)\load$.

We introduce the following Lyapunov function candidate, with parameters $\ep{1}, \ep{2} > 0$ to be determined later.
Note that \eqref{e:lyap1} below is an energy-based storage function commonly used in the study of the class of incrementally passive systems \citep{persis15bregman}, while the addition of \eqref{e:lyap2} will ensure that \lyap is strictly decreasing along any solution to \eqref{e:incrementalswing} other than the optimal equilibrium $(\bar\aedge, \0, \bar\ctrl)$:
\begin{subequations} \label{e:lyap}
\begin{align}
\begin{split} \label{e:lyap1}
\lyap(\aedge, \freq, \ctrl)
  &= \pot(\aedge) - \pot(\bar \aedge) - \nabla \pot(\bar \aedge)\T (\diff\aedge) \\
  &+ \frac12 \freq\T \inert \freq + \frac12 (\diff\ctrl)\T \balance (\diff\ctrl)
\end{split} \\
\begin{split} \label{e:lyap2}
  &+ \ep{1} (\potdiff)\T \balance \inert \freq \\
  &- \ep{2} (\diff{\ctrl})\T \ones \inert \freq.
\end{split}
\end{align}
\end{subequations}
The cross-terms allow us to prove exponential convergence to the equilibrium. The need for two separate cross-terms will become clear in Remark \ref{rem:2-cross-terms} on page~\pageref{rem:2-cross-terms}.

Note that \lyap vanishes at the equilibrium $(\bar\aedge, \0, \bar\ctrl)$ of \eqref{e:swing}.
In addition, we have the following Lemma.
\begin{lemma} \label{lem:lyapbound}
Suppose Assumption \ref{asm:existence} holds.
There exist sufficiently small $\ep{1}, \ep{2}$ and positive constants \lowerbd, \upperbd such that for all \aedge with $\incidence\T \aedge\in\rotorset$, we have
\[ \lowerbd \|\statevec\gen(\aedge, \freq\gen, \ctrl)\|^2 \leq \lyap(\aedge, \freq, \ctrl) \leq \upperbd \|\statevec\gen(\aedge, \freq\gen, \ctrl)\|^2, \]
where $\statevec\gen(\aedge, \freq\gen, \ctrl) := \col(\diff\aedge, \freq\gen, \diff\ctrl)$.
\end{lemma}

See the Appendix for this Lemma's proof.

For ease of the notation, we will omit the explicit parameters of $\statevec\gen$ in the rest of the paper.

\begin{remark}
Note that $\lyap(\aedge, \freq, \ctrl)$ does not explicitly depend on $\freq\load$, and thus 
$\freq\load$ does not appear in the lower and upper bounds of $\lyap$. 
\end{remark}

\subsection{Derivative of the Lyapunov function} \label{ss:lyapderiv} 

Before proving that the Lyapunov function strictly decreases along the solutions of the system, we need to perform an additional change of coordinates on the state components of the controller.

Let $\filter \in \R^{\cardin \times (\cardin-1)}$ be a matrix of orthonormal columns, orthogonal to $\frac{1}{\sqrt\ctrlm} \balance^{-\frac12} \1$, where $\mu = \1\T \balance\I \1$.
Then the new controller coordinates are defined as follows:
\begin{equation} \label{e:transform}
\begin{aligned}
\begin{bmatrix} \ctrld \\ \ctrla \end{bmatrix} 
  &= \begin{bmatrix} \filter\T \\ \frac{1}{\sqrt\ctrlm} \1\T \balance^{-\frac12} \end{bmatrix} 
     \balance^{\frac12} \ctrl 
  &&=: \filterall\I \ctrl \\
\ctrl &= \balance^{-\frac12}
     \begin{bmatrix} \filter & \frac{1}{\sqrt\ctrlm} \balance^{-\frac12} \1 \end{bmatrix}
     \begin{bmatrix} \ctrld \\ \ctrla \end{bmatrix}
  &&=: \filterall
     \begin{bmatrix} \ctrld \\ \ctrla \end{bmatrix}
\end{aligned}
\end{equation}

In these coordinates, the dynamics \eqref{e:incrementalswing} takes the form
\begin{align} \label{e:transincrementalswing}
\begin{split}
\dot\aedge   &= \average \freq \\
\inert \dot\freq
             &= -\damp \freq - (\potdiff) \\ &\quad+ \frac{1}{\sqrt\ctrlm} \balance\I \1 (\diff\ctrla) + \balance^{-\frac12} \filter (\diff\ctrld) \\
\dot \ctrld       &= -\filter\T \balance^{\frac12} \laplcomm \balance^{\frac12} \filter (\diff\ctrld) - \filter\T \balance^{-\frac12} \freq \\
\dot \ctrla   &= \frac{-1}{\sqrt\ctrlm} \1\T \balance\I \freq,
\end{split}
\end{align}
and the Lyapunov function writes as
\begin{multline} \label{e:translyap}
\lyap(\aedge, \freq, \ctrld, \ctrla)
   = \pot(\aedge) - \pot(\bar \aedge) - \nabla \pot(\bar \aedge)\T (\diff\aedge) \\
\begin{aligned}
  &+ \frac12 \freq\T \inert \freq + \frac12 (\diff \ctrld)\T (\diff \ctrld)
   + \frac12 (\diff\ctrla)^2 \\
  &+ \ep{1} \freq\T \inert \balance (\potdiff) \\
  &- \ep{2} \frac1{\sqrt\ctrlm} \freq\T \inert \1 (\diff\ctrla),
\end{aligned}
\end{multline}
where by an abuse of the notation we are denoting the Lyapunov function in the new coordinates with the same symbol $\lyap$ as before.

To prove that $\lyap(\aedge, \freq, \ctrld, \ctrla)$ is strictly decreasing along solutions of 
\eqref{e:lyapmat},
we must compute its directional derivative along the vector field defined by the right-hand side of \eqref{e:transincrementalswing} and show that it is strictly negative.

For ease of notation, we define
\[
\dot \lyap(\aedge, \freq, \ctrld, \ctrla) = \frac{\partial \lyap}{\partial (\aedge, \freq, \ctrld, \ctrla)}
\begin{bmatrix}
\dot \aedge \\ \dot \freq\\ \dot \ctrld \\ \dot \ctrla
\end{bmatrix}
\]
where the vector of derivatives on the right-hand side are associated with the vector field  \eqref{e:transincrementalswing}.
\begin{equationfloat*}
\begin{equation} \label{e:lyapmat}
\lyapmat(\aedge) = \symm\begin{bmatrix}
       \ep{1} \balance &
         \ep{1} \balance \damp &
           -\ep{1} \balance^{\frac12} \filter &
             \0 \\
       \0 &
        \damp - \ep{1} \inert \balance \nabla^2\pot(\aedge) 
              - \ep{2} \frac{1}{\ctrlm} \inert \ones \balance\I &
           \0 &
             -\ep{2} \frac1{\sqrt\ctrlm} \damp \1 \\
       \0 & 
         \0 &
           \filter\T \balance^{\frac12} \laplcomm \balance^{\frac12} \filter &
             \ep{2} \frac1{\sqrt\ctrlm} \filter\T \balance^{-\frac12} \1 \\
       \0 &
         \0 &
           \0 &
             \ep{2}
\end{bmatrix}\!.
\end{equation}
\end{equationfloat*}

\begin{lemma} \label{lem:lyapderiv}
The directional derivative of $W$ along the vector field \eqref{e:transincrementalswing} satisfies
\[ \dot \lyap(\aedge, \freq, \ctrld, \ctrla) = - \lyapvec\T \lyapmat(\aedge) \lyapvec, \]
with $\lyapmat(\aedge)$ as in \eqref{e:lyapmat}, and
\[ \lyapvec(\aedge, \freq, \ctrld, \ctrla) := \col(\potdiff, \freq, \diff\ctrld, \diff\ctrla). \]
\end{lemma}
As with $\lyapvec\gen$, we omit the parameters of \lyapvec in the following.

\begin{proof*}{Proof}
Using \eqref{e:transincrementalswing}, the part of $\lyap(\aedge, \freq, \ctrld, \ctrla)$ which is independent of $\ep{i}$ has the  derivative
\begin{align*}
  &\freq\T (-\damp \freq + \balance^{-\frac12} \filter (\diff \ctrld) + \ctrlm^{-\frac12} \balance\I \1 (\diff\ctrla)) \\
  &\quad+ (\diff \ctrld)\T(- \filter\T \balance^{\frac12} \laplcomm \balance^{\frac12} \filter (\diff \ctrld) - \filter\T \balance^{-\frac12} \freq) \\
  &\quad+ (\diff \ctrla) (- \ctrlm^{-\frac12} \1\T \balance\I \freq) \\
={}& -\freq\T \damp \freq - (\diff \ctrld)\T\filter\T \balance^{\frac12} \laplcomm \balance^{\frac12} \filter (\diff \ctrld).
\end{align*}
Here, we used the fact that $\incidence\T \average = \incidence\T$ to cancel the $(\potdiff)\T\freq$--terms.
Meanwhile, the derivative of the first cross-term (ignoring \ep{1}) is
\begin{align*}
  &\freq\T \inert \balance \nabla^2\pot(\aedge) \freq \\
  &\qquad - (\potdiff)\T \balance \damp \freq \\
  &\qquad - (\potdiff)\T \balance (\potdiff) \\
  &\qquad + (\potdiff)\T \balance^{\frac12} \filter (\diff \ctrld).
\end{align*}
Noting that $\col(\inert\gen, \0) \freq\gen = \inert \col(\freq\gen, \freq\load) = \inert \freq$, the derivative of the second cross-term, ignoring \ep{2}, is
\begin{align*}
  &\ctrlm\I \freq\T \inert \ones \balance\I \freq 
   - \ctrlm^{-\frac12} (\diff\ctrla) \1\T ( -\damp \freq \\
  &\quad + \balance^{-\frac12} \filter (\diff \ctrld) + \ctrlm^{-\frac12} \balance\I \1 (\diff\ctrla)) \QEDhere
\end{align*}
\end{proof*}

Having computed the directional derivative $\dot\lyap(\aedge, \freq, \ctrld, \ctrla)$, we now show that it is strictly negative.
\begin{lemma} \label{lem:lyapmatpos}
Suppose that the communication graph is connected.
Then, there exist sufficiently small values of $\ep{1}$ and $\ep{2}$ such that $\lyapmat(\aedge) > 0$ for all \aedge with $\incidence\T \aedge \in \rotorset$.
\end{lemma}

\begin{proof}
For notational convenience, we will refer to $\lyapmat(\aedge)$ simply as \lyapmat.
First, we reduce \lyapmat to a block diagonal form $\lyapmat'$ using Lemma~\ref{lem:removecrossterms} in Appendix~\ref{sec:techlemmas}.
Then we discuss the blocks of $\lyapmat'$.

\emph{Reduction to a block diagonal form.}
To reduce \lyapmat to block diagonal form, we apply Lemma~\ref{lem:removecrossterms} two times.
First, we express the matrix $K$ as the sum
\[ \lyapmat = \ep{1} \lyapmat_{\ep{1}} +\ep{2} \lyapmat_{\ep{2}} + 
     \bdiag(\0, \damp, \filter\T \balance^{\frac12} \laplcomm \balance^{\frac12} \filter, 0). \]
Then, we focus on the $\ep{1}$-terms.
\begin{equation*}
\lyapmat_{\ep{1}} = \symm\left[\begin{array}{@{}c:ccc@{}}
       \balance &
         \balance \damp &
           -\balance^{\frac12} \filter &
             \0 \\
\hdashline \\[-1em]
       \0 &
         -\inert \balance \nabla^2\pot(\aedge) &
           \0 &
             \0 \\
       \0 & \0 & \0 & \0 \\
       \0 & \0 & \0 & 0
\end{array}\right]\!.
\end{equation*}
Using the partition indicated with $b\T = \frac12 \begin{bmatrix}\balance^{\frac12}& \balance^{\frac12}& \0\end{bmatrix}$ and $c = \bdiag(\balance^{\frac12} \damp, -\filter, \0)$ yields $\lyapmat_{\ep{1}} \geq \lyapmat_{\ep{1}}'$, with
\begin{multline*}
K_{\ep{1}}' = \bdiag\bigg(\frac12 \balance, \\
    -\symm(\inert \balance \nabla^2\pot(\aedge)) - \damp \balance \damp,
    -I_{\cardin-1},
    0 \bigg).
\end{multline*}

Next, we do the same for the $\ep{2}$-terms.
\begin{equation*}
\lyapmat_{\ep{2}} = \symm\left[\begin{array}{@{}ccc:c@{}}
       \0 & \0 & \0 & \0 \\
       \0 &
        - \frac{1}{\ctrlm} \inert \ones \balance\I &
           \0 &
             -\frac1{\sqrt\ctrlm} \damp \1 \\
       \0 & 
         \0 &
           \0 &
             \frac1{\sqrt\ctrlm} \filter\T \balance^{-\frac12} \1 \\
\hdashline \\[-1em]
       \0 &
         \0 &
           \0 &
             1
\end{array}\right]\!.
\end{equation*}

This time, we choose $b = \sqrt{\frac{3\cardin}{\ctrlm}} \bdiag(\0, -\damp, \balance^{\frac12}\filter)$ and $c = \frac1{2\sqrt{3\cardin}} \1_{3\cardin}$.
This yields $\lyapmat_{\ep{2}} \geq \lyapmat_{\ep{2}}'$, with
\begin{multline*}
\lyapmat_{\ep{2}}' = \bdiag\bigg(\0,
    -\frac{1}{\ctrlm} \symm(\inert \ones \balance\I) - \frac{3\cardin}{\ctrlm} \damp^2, \\
    -\frac{3\cardin}{\ctrlm} \filter\T \balance\I \filter,
    1 - \frac14 \bigg).
\end{multline*}

The terms independent of $\ep{1}$ and $\ep{2}$ are already in block diagonal form, and strictly positive definite.
Hence, we let
\begin{multline}
\lyapmat' = \bdiag\bigg(
    \frac12 \ep{1} \balance, \\
\begin{aligned}
    &\damp
      - \ep{1} \symm( \inert \balance \nabla^2\pot(\aedge) )
      - \ep{1} \damp \balance \damp \\
      &\phantom{\damp}
      - \ep{2} \frac{1}{\ctrlm} \symm( \inert \ones \balance\I )
      - \ep{2} \frac{3\cardin}{\ctrlm} \damp^2, \\
    &\filter\T \balance^{\frac12} \laplcomm \balance^{\frac12} \filter
      - \ep{1} I_{\cardin-1}
      - \ep{2} \frac{3\cardin}{\ctrlm} \filter\T \balance\I \filter, \\
    &\frac34 \ep{2}
\bigg),
\end{aligned}
\end{multline}
and conclude that $\lyapmat \geq \lyapmat'$.

\emph{Positive definiteness.}
We note that $\damp > 0$.
Also, since all columns of \filter are perpendicular to $\balance^{-\frac12}\1$, it is clear that $\filter\T \balance^{\frac12} \laplcomm \balance^{\frac12} \filter > 0$ given that the communication graph is connected.
Finally, by definition, $\balance > 0$.
At this point, $\lyapmat \geq \lyapmat' > 0$ provided that $\ep{1}$ and $\ep{2}$ are chosen sufficiently small.
\end{proof}

\begin{remark}[Purpose of the cross-terms] \label{rem:2-cross-terms} \sloppy
Note that the role of the cross-terms in \lyap is now clear: each serves to make one block of $\lyapmat'$ strictly positive definite, at a slight cost to the blocks belonging to the quadratic terms in \freq and $\diff\ctrld$.
\end{remark}

The coordinate transformation has served its purpose in the proof of Lemma~\ref{lem:lyapmatpos}, and now we go back to the original coordinates for the next part. To this end, let
\[ \statevec(\aedge, \freq, \ctrl) := \col(\diff\aedge, \freq, \diff\ctrl). \]

We formalize the results proved so far in a statement. 

\begin{proposition} \label{prp:lyapbound}
Suppose Assumption~\ref{asm:existence} holds.
There exist sufficiently small $\ep{1}, \ep{2}$ and a positive constant $\bound$ such that for any \aedge such that $\incidence\T \aedge\in\rotorset$, and any $\freq\gen, \ctrl$, the directional derivative of \lyap along the vector field \eqref{e:swing}--\eqref{e:controller} satisfies
\begin{equation*}
\dot\lyap(\aedge, \freq, \ctrl) \leq - \bound \lyap(\aedge, \freq, \ctrl).
\end{equation*}
\end{proposition}

\begin{proof}
It is clear from the proof of Lemma~\ref{lem:lyapmatpos} that, by the invariance of the directional derivative with respect to the change of coordinates, the Lyapunov function $\lyap$ written in the original controller coordinates $\ctrl$ satisfies the same inequality with the variables $\col(\ctrld, \ctrla)$ replaced by $\ctrl$.

Given Lemma~\ref{lem:lyapmatpos}, $\lyapmat(\aedge) > 0$ for all \aedge such that $\incidence\T \aedge \in\rotorset$.
Since this is a closed set, there exists a positive constant \derivbd such that $\lyapmat(\aedge) > \derivbd I$.
Given this, $\dot\lyap(\aedge, \freq, \ctrl) \leq - \derivbd \|\lyapvec(\aedge, \freq, \ctrl)\|^2$.

The first statement of Lemma~\ref{lem:activepower} provides that $\|\potdiff\|^2 \geq \al{1} \|\diff\aedge\|^2$.
Hence, $\|\lyapvec\|^2 \geq \min(\al{1},1) \|\statevec\|^2$, and $\dot\lyap(\aedge, \freq, \ctrl) \leq - \derivbd \min(\al{1},1) \|\statevec(\aedge, \freq, \ctrl)\|^2$.

Hence we remark that
\begin{align} \dot\lyap(\aedge, \freq, \ctrl)
 &\leq - \derivbd\min(\al{1},1) \|\statevec(\aedge, \freq, \ctrl)\|^2 \nonumber\\
 &\leq - \derivbd\min(\al{1},1) \|\statevec\gen(\aedge, \freq\gen, \ctrl)\|^2 \nonumber\\
 &\leq - \frac{\derivbd\min(\al{1},1)}{\upperbd} \lyap(\aedge, \freq, \ctrl) \nonumber\\
 &  =: - \bound \lyap(\aedge, \freq, \ctrl). \QEDhere
\end{align}
\end{proof}

\subsection{Exponential convergence to the equilibrium} \label{sec:expstab} %

Having shown that the directional derivative of $\lyap(\aedge, \freq, \ctrl)$ is strictly negative along the vector field of the closed-loop system, we show exponential convergence to
the equilibrium.

\begin{theorem} \label{thm:expstab}
Suppose Assumption \ref{asm:existence} holds.
There exists a neighborhood of the equilibrium $(\bar\aedge, \0, \bar\ctrl)$ such that all the solutions of the closed-loop system \eqref{e:swing}--\eqref{e:controller} that start from that neighborhood converge exponentially to the equilibrium, i.e.\ there exist positive scalars $\alpha, \beta$ such that for all $t \geq 0$,
\[ \|\statevec(t)\| \leq \alpha \|\statevec(0)\| \exp{-\beta t}, \]
with $\statevec(\aedge, \freq, \ctrl) = \col(\diff\aedge, \freq, \diff\ctrl)$.
\end{theorem}

\begin{proof}
The equilibrium $(\bar \aedge, \0, \bar \ctrl)$ is a strict minimum of $\lyap(\aedge, \freq, \ctrl)$ by Lemma~\ref{lem:lyapbound}.
Therefore there exists a compact level set $\levelset$ around $(\bar \aedge, \0, \bar \ctrl)$.
Moreover, without loss of generality, any point on the level set $\levelset$ is such that $\incidence\T \aedge \in \rotorset$.
Hence, by Proposition~\ref{prp:lyapbound}, $\dot \lyap \leq -\bound \lyap \leq 0$ along the solutions of the closed loop system, which shows the invariance of $\levelset$.
Integrating this inequality between 0 and $t$ and applying Lemma~\ref{lem:lyapbound} yields exponential convergence of the state variables $\diff\aedge, \freq\gen, \diff\ctrl$ to the origin, namely 
\begin{align*}
\lyap(\aedge(t), \freq(t), \ctrl(t)) &\leq \lyap(\aedge(0), \freq(0), \ctrl(0)) \exp{-\bound t} \\
\|\statevec\gen(t)\|^2 &\leq \frac{\upperbd}{\lowerbd} \|\statevec\gen(0)\|^2 \exp{-\bound t}.
\end{align*}

Now, by Claim \ref{enumitem:lem-activepower-bound-genonly} of Lemma~\ref{lem:activepower}, we also have $\|\statevec(t)\|^2 \leq \gamma \|\statevec\gen(t)\|^2$ for some positive scalar $\gamma$.
Since the right-hand side is converging exponentially to zero, so is $\statevec(t)$, since  $\damp\load$ is positive definite.  We conclude that the full state $(\diff\aedge, \freq, \diff\ctrl)$ exponentially converges to the origin as claimed, with $\alpha = \sqrt{\gamma \upperbd/\lowerbd}$ and $\beta = \frac12 \bound$.  
\end{proof}

\section{Convergence bounds under DoS} \label{sec:dos} 

In the previous sections, we have quantified the convergence rate of solutions to \eqref{e:swing} in closed loop with the controllers \eqref{e:controller}.
We will now consider the effect of a DoS event, which interrupts the communication between controllers as detailed in Assumption~\ref{asm:doslimits} below.
{We conclude, by characterizing the parameters of DoS for which the closed loop system retains exponential convergence to the optimal synchronous solution \eqref{e:feas}.}

\subsection{Intermittent feedback measurements} \label{ss:dosintro} 
In the current setting, we consider the case in which the communication graph is disrupted.
To quantify the impact of this disruption on performance, we consider the worst-case scenario in which all communication links fail simultaneously during the disruption period {\citep{senejohnny15selftriggered}.}
Without communication, the controllers will still ensure that $\freq \to \0$, but can no longer guarantee economic optimality \citep[Remark 6]{trip16internal} and are vulnerable to noise in measurements \citep{andreasson14distributed}.

{In the presence of communication disruptions, the system evolves according to the following two modes:}
\begin{enumerate}
\item the nominal {mode}, in which the system and controllers {obey the dynamics \eqref{e:swing}, \eqref{e:controller}}   as detailed previously;
\item the denial-of-service (DoS) {mode}, {in which the system evolves according to \eqref{e:swing}, \eqref{e:controller}  with } $\laplcomm = \0{_{\cardin \times \cardin}}$ in \eqref{e:controller}.
\end{enumerate}
\begin{remark}
Notice that a third state is possible, in which a subset of the communication links is interrupted.
While our results continue to hold for this case, the conditions derived, namely Theorem~\ref{thm:expstabdos}, turn out to be conservative.
A way to reduce this conservatism is to exploit the notion of persistency of communication inspired by \cite{senejohnny16jamming,arcak07passivity}.
This study will be pursued in a future work.
\end{remark}

The system under consideration can now be formalized as follows \citep{persis14resilient-ifac}.
Let $\dosstart_i \geq 0$ denote the starting time of the $i^{\rm th}$ DoS failure, i.e.\ the time of $i^{\rm th}$ DoS transition from inactive to active.
Furthermore, let $\dosduration_i > 0$ denote the length of the $i^{\rm th}$ DoS failure, such that $\dosstart_i + \dosduration_i < \dosstart_{i+1}$.
We then denote the $i^{\rm th}$ DoS interval by $\dositv_i := [\dosstart_i, \dosstart_i + \dosduration_i)$.
During these intervals, no communication is possible between the controllers.
The choice of these intervals is not allowed to be completely arbitrary; limiting the duration of the failure is necessary for closed-loop stability to be achievable at all.
In this light, the DoS failure is restricted as follows.

Given a sequence {of DoS intervals} $\{\dositv_i, i = 1, \ldots, \cardindos\}$, let 
\[ \dosallitv(t) := \bigcup_{i = 1}^\cardindos \dositv_i \cap [0,t] \]
denote the union of DoS intervals up to time $t$.
\begin{assumption} \label{asm:doslimits} \citep[Assumption 1]{persis14resilient-ifac}
{There exist constants $\kappa\in \mathbb{R}_{>0}$ and $\tau\in \mathbb{R}_{>1}$}
such that for all $t \geq 0$,
\begin{equation} \label{e:doslen}
\doslen \leq \dosboundstart + \frac{t}{\dosboundrate}.
\end{equation}
\end{assumption}

The rationale behind this inequality is that, if $\dosboundstart = 0$, the DoS failure is active at most a proportion of ${1}/{\dosboundrate}$ of the time (since $\dosboundrate > 1$).
Adding \dosboundstart is necessary, since if $\dosstart_0 = 0$, $|\dosallitv(\dosduration_0)| = \dosduration_0 \geq \dosduration_0 / \dosboundrate$, hence $\dosduration_0$ is required to be zero.
The addition of $\dosboundstart > 0$ therefore allows the failure to be active at the start of the interval under consideration.

No further conditions are placed on the structure of the DoS state, allowing it to occur aperiodically, allowing subsequent events to differ in length, and allowing any or no specific stochastic distribution \citep{persis14resilient-ifac,persis15input}.

\subsection{Exponential convergence under DoS}

To prove the main result of this section, we first state the existence of an exponential growth during DoS intervals.

\begin{proposition} \label{prp:lyapbound-unstable}
Let Assumption~\ref{asm:existence} hold.
There exist sufficiently small $\ep{1}, \ep{2}$ and a positive constant $d$ such that for any \aedge for which $\incidence\T \aedge\in\rotorset$, and any $\freq\gen, \ctrl$, the directional derivative of $\lyap(\aedge, \freq, \ctrl)$ along the vector field  \eqref{e:swing}, \eqref{e:controller} \emph{with} $\mathcal{L}_\xi=\mathbb{0}_{n\times n}$ satisfies:
\begin{equation*}
\dot\lyap(\aedge, \freq, \ctrl) \leq d\, \lyap(\aedge, \freq, \ctrl).
\end{equation*}
\end{proposition}

\begin{proof}
By a minor variation of Lemma~\ref{lem:lyapderiv} and Lemma~\ref{lem:lyapmatpos}, one writes
\[ \dot\lyap(\aedge, \freq, \ctrl) \leq \bound_{\rm DoS} \|\lyapvec\|^2 \]
for 
\[ \bound_{\rm DoS} := -\min_{\incidence\T \aedge \in \rotorset} \lmin{\lyapmat(\aedge)|_{\laplcomm=\0}} \]
positive for positive values of \ep{1} and \ep{2}.
From Lemma~\ref{lem:activepower}, one obtains a positive scalar $\al{2}$ such that $\dot\lyap(\aedge, \freq, \ctrl) \leq \bound_{\rm DoS} \max(1, \al{2}) \|\statevec\|^2$.
To proceed, we apply Claim \ref{enumitem:lem-activepower-bound-genonly} of Lemma~\ref{lem:activepower} to see that $\|\statevec\|^2 \leq \gamma \|\statevec\gen\|$ for a positive scalar $\gamma$.
Finally, we apply Lemma~\ref{lem:lyapbound} to end up at the claim of the Theorem:
\begin{align*}
\dot\lyap(\aedge, \freq, \ctrl) &\leq \bound_{\rm DoS} \max(1, \al{2}) \gamma \|\statevec\gen\|^2 \\
  &\leq \frac{\bound_{\rm DoS}}{\lowerbd} \max(1, \al{2}) \gamma \lyap(\aedge, \freq, \ctrl). \QEDhere
\end{align*}
\end{proof}

We are now ready to state the main result of this section. It applies to the solutions of system \eqref{e:swing}
controlled by 
\begin{align} \label{e:controller-switched}
\dot \ctrl &= - \laplcomm(t) \balance \ctrl - \balance\I \freq &  \inp &= \ctrl,
\end{align}
where 
\[
\laplcomm(t) = \left\{
\begin{array}{ll}
\laplcomm & t \not\in  \Xi(t)\\
\mathbb{0}_{n\times n} & t\in \Xi(t).
\end{array}
\right.
\]

\begin{theorem} \label{thm:expstabdos}
Let Assumption \ref{asm:existence} hold, and let $\bound, d$ be as in Propositions~\ref{prp:lyapbound} and~\ref{prp:lyapbound-unstable}, respectively. 
Suppose that the communication between the controllers is subject to a DoS event, for which Assumption~\ref{asm:doslimits} holds with
\begin{equation}\label{e:tau-DoS}
\tau > 1 + \frac{d}{c}.
\end{equation}
Then, there exists a neighborhood of the equilibrium $(\bar\aedge, \0, \bar\ctrl)$ such that solutions of the closed-loop system \eqref{e:swing}, \eqref{e:controller-switched}, that start from this neighborhood exponentially converge to the equilibrium, namely, for all $t \geq 0$ we have
\begin{equation} \label{e:dos-exp-decline}
  \|\statevec(t)\| \leq \alpha \exp{-\beta t} \|\statevec(0)\|,
\end{equation}
with $\beta = \frac{1}{2}(c- \frac{c+d}{\tau}) > 0$, $\alpha = \sqrt{\gamma e^{\dosboundstart (\bound + d)} \upperbd / \lowerbd}$, and $\gamma$ as in  Lemma~\ref{lem:activepower}.
\end{theorem}

\begin{proof}
First, we note that the equilibrium $(\bar\aedge, \0, \bar\ctrl)$ of system \eqref{e:swing}, \eqref{e:controller-switched}, is Lyapunov stable \citep[e.g.][]{depersis16lyapunov,trip16internal}. In fact, the function $W$ in \eqref{e:lyap} with $\ep{1} = \ep{2} = 0$, provides a common weak Lyapunov function for the switched system \eqref{e:swing}, \eqref{e:controller-switched}.
Hence, there exists a neighborhood of the equilibrium point for which any solution that originates in it remains in the set of points such that $\incidence\T \aedge\in \rotorset$.
Then, for all $t \geq 0$,  
\begin{multline} \label{e:exp-bound-dos}
\lyap(\aedge(t), \freq(t), \ctrl(t)) \\
	\leq \lyap(\aedge(0), \freq(0), \ctrl(0))
	     \exp{(c+d)\kappa} \exp{-t(c -\frac{c+d}{\tau})},
\end{multline}
where, to derive the inequality, we have distinguished in $[0, t]$ between intervals during which \lyap exponentially decays with rate $c$ (DoS-free intervals) and intervals during which \lyap exponentially increases with rate $d$ (DoS intervals), and used Propositions~\ref{prp:lyapbound} and~\ref{prp:lyapbound-unstable}. Therefore, using Lemma \ref{lem:lyapbound},
\[
\|\statevec\gen(t)\| \leq \sqrt{\frac{\upperbd}{\lowerbd}} \exp{(c+d)\frac{\kappa}{2}}
\exp{-\frac{t}{2}(c -\frac{c+d}{\tau})} \|\statevec\gen(0)\|.
\]
This results in exponential convergence of $\statevec\gen(t)$ with $\tilde\alpha := \sqrt{\upperbd/\lowerbd} \, \exp{(c+d)\frac{\kappa}{2}}$ and $\beta=\frac{1}{2}(c -\frac{c+d}{\tau})$.
Note that $\beta>0$ by \eqref{e:tau-DoS}.

Finally, setting $\alpha := \sqrt\gamma \tilde\alpha$ using Claim~\ref{enumitem:lem-activepower-bound-genonly} of Lemma~\ref{lem:activepower}, we conclude that the full state $(\diff\aedge, \freq, \diff\ctrl)$ exponentially converges to the origin as claimed.
\end{proof}

The result of theorem above indicates that optimal resource allocation and exponential convergence are preserved if the proportion of time for which the DoS is active is sufficiently small, see \eqref{e:tau-DoS}.
Moreover, the obtained exponential convergence directly relates bounds on the behavior of the closed loop power network, specifically the overshoot $\alpha$ and convergence rate $\beta$, to a combination of the physical and cyber parameters of the system and the ongoing DoS event.
This quantifies the performance degradation of the system as a result of the disruption.

\section{Simulations} \label{sec:simulations} 

\begin{figure}[t]\centering
\begin{tikzpicture}
	\node [draw, circle] (area1) at (0, 0) {Node 1};
	\node [draw, circle] (area2) at (2, 0) {Node 2};
	\node [draw, circle] (area3) at (2,-2) {Node 3};
	\node [draw, circle] (area4) at (0,-2) {Node 4};
	\draw (area1) -- (area2);
	\draw (area2) -- (area3);
	\draw (area2) -- (area4);
	\draw [rounded corners=2ex, blue, dashed] (area1) -- ( .5,-1  ) -- (area4);
	\draw [rounded corners=2ex, blue, dashed] (area2) -- (1.5,-1  ) -- (area3);
	\draw [rounded corners=2ex, blue, dashed] (area3) -- (1,  -1.5) -- (area4);
	\draw [blue, dashed] (area1) -- (area3);
	\node at ( 1, 1) {$\suscept_{12} = 25.6$};
	\node at (-1,-1) {$\suscept_{24} = 21.0$};
	\node at ( 3,-1) {$\suscept_{23} = 33.1$};
\end{tikzpicture}
\caption{The node network used for the simulations.
Solid lines denote the transmission lines, while dashed blue lines represent the communication graph used by the controllers.}
\label{f:examplenet}
\end{figure}
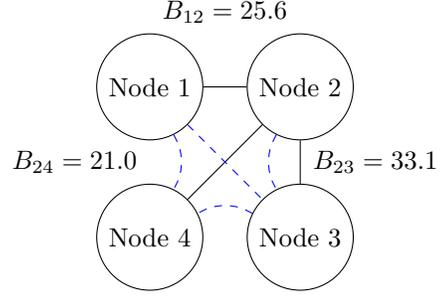

\begin{table}[t]\centering\vspace{1ex}
\captionsetup{width=\linewidth} 
\caption{Numerical values for the simulation}
\begin{subtable}{\linewidth}
\caption{Parameter values for the simulations.
         All parameters are provided in `per unit'.}
\label{t:exampleparams}
\centering
\begin{tabular}{lrrrrr}
\toprule
Node $i$              & Node 1   & Node 2   & Node 3   & Node 4 \\
\midrule
$\inert_{i}$          & $  3.26$ & $  3.26$ &   (load) &   (load) \\
$\damp_{i}$           & \multicolumn{4}{c}{All equal to 1}        \\
$\voltage_{i}$        & $  0.98$ & $  0.97$ & $  0.96$ & $  1.04$ \\
$\suscept_{ii}$       & $-46.60$ & $-79.70$ & $-33.10$ & $-21.00$ \\
$\balance_{i}$        & $  1.00$ & $  0.75$ & $  1.50$ & $  0.50$ \\
$\demand_{i}$         & $  0   $ & $  0   $ & $  0.72$ & $  0.24$ \\
\bottomrule
\end{tabular} \vspace{1em}

\end{subtable}
\begin{subtable}{\linewidth}
\caption{Convergence values resulting from the parameters of the case study.}
\label{t:simvals}
\centering
\begin{tabular}{r@{${}={}$}lr@{${}={}$}l}
\toprule
$\ep{1}$       & \num{0.025}   & $\ep{2}$       & \num{0.030} \\
$\lowerbd$     & \num{0.010}   & $\upperbd$     & \num{6.073} \\
$\derivbd$     & \num{0.012}   & $\bound$       & \num{4.120e-4} \\
$\alpha$       & \num{173.5}   & $\beta$        & \num{1.291e-4} \\
$\kappa$       & \num{10}      & $\tau$         & \num{1.5} \\
\bottomrule
\end{tabular}
\end{subtable}
\end{table}

To illustrate the effect of interrupted communication, we simulate the action of the controllers, along with the values of \lyap, on an academic example of an electricity grid, taken from \citet{trip16internal}.
The network contains four nodes, connected by the graph depicted in Figure \ref{f:examplenet}.
Two nodes are generators, two nodes are loads.
The parameter values are listed in Table \ref{t:exampleparams}.

The network was first initialized to a steady state with load profile $\demand = \0$.
At $t = 0$, the profile was changed to the values in Table \ref{t:exampleparams}, and the system was subjected to a DoS sequence.
This initialization ensures that controller communication is essential for the system to reach the optimal values for $\ctrl$ and $\aedge$ given by \eqref{e:feas}.
The sequence starts with approximately \SI{30}{s} of DoS, and then short intervals of communication as dictated by \eqref{e:doslen}.

In Figure \ref{f:plot}, the evolution of $\|\statevec(\aedge, \freq, \ctrl)\|$ during the simulation is shown.
It is upper bounded as in \eqref{e:dos-exp-decline} in Theorem~\ref{thm:expstabdos}; we illustrate the slope of the bound using the red curve in the Figure.
The bound itself is less tight due to the large value of $\alpha$ from \eqref{e:dos-exp-decline}.
The numerical values of the parameters relating to convergence are displayed in Table \ref{t:simvals}.

\section{Conclusions} \label{sec:conclusions} 

We have introduced a Lyapunov function to show exponential convergence of power networks under the distributed averaging integral controllers from e.g., \citet{dorfler14breaking, trip16internal, monshizadeh16agreeing}, and, as an academic application, studied their performance when their communication network is intermittently interrupted.
We have derived a bound on the decay rate of the solutions in terms of properties of the interruption sequence.

Disruptions of other natures can be considered; sophisticated adversaries may opt to delay the communication signal or even inject false measurements.
Future work will quantify robustness to such measurement errors. We believe that the Lyapunov function introduced in the paper is very useful to study robustness to sensor noises \citep{andreasson14distributed}.
Also, this work considers only the case where communications are entirely removed; it is very interesting to consider disruptions of a subset of the communication links as in e.g.\ \citet{senejohnny16jamming}.

In addition to power networks, distributed averaging controllers arise in several other domains, such as distributed optimization. In that context, an exponential Lyapunov function could be useful to characterize the convergence speed as an alternative to heavy ball methods \citep{polyak16optimization}.

Finally, it would be interesting to investigate possible connections of the results
in this paper with the quadratic Lyapunov functions and resilience certificates of \citet{turitsyn17framework}.

\begin{figure}[t]\centering
	\includegraphics[width=\columnwidth]{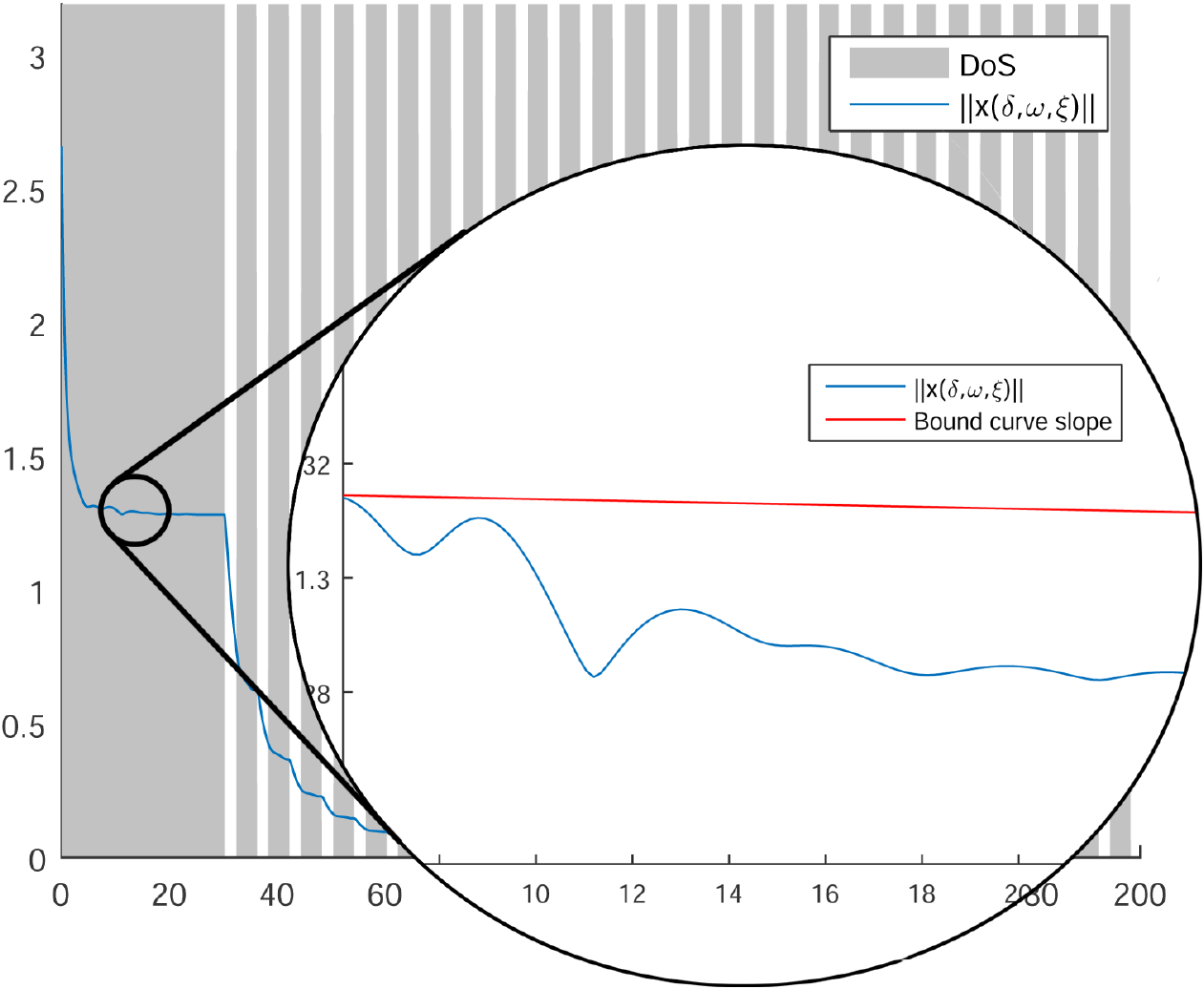}
	\caption{Evolution of $\|\statevec(\aedge, \freq, \ctrl)\|$ during the simulation. 
	         Shaded vertical bars represent the times during which controller communication was unavailable.
	         The detailed view illustrates the tightness of the decay rate $\beta$ obtained in Theorem~\ref{thm:expstabdos}.}
	\label{f:plot}
\end{figure}

\appendix
\section{Proofs and technical lemmas} \label{sec:techlemmas} 

\begin{proof*}{Proof of Lemma~\ref{lem:lyapbound}}
Note that at the equilibrium $(\aedge, \freq\gen, \ctrl) = (\bar\aedge, \0, \bar\ctrl)$, $\lyap(\aedge, \freq, \ctrl)$ and $ \statevec\gen$ are both zero, and the inequalities in the lemma trivially hold.

To show the existence of the lower and upper bounds, we will first investigate the terms of $\lyap(\aedge, \freq, \ctrl)$ in \eqref{e:lyap1}.
This will lead to initial estimates for the bounds of the entirety of $\lyap(\aedge, \freq, \ctrl)$.
Then, by an appropriate choice of the $\ep{i}$ occurring in \eqref{e:lyap2}, we will limit and quantify the deviation from these estimates caused by the cross-terms.

Consider the terms in \eqref{e:lyap1}.
Since $\inert\gen$ and $\balance$ are diagonal matrices with positive elements, outside equilibria we have
\[ 0 < \lmin{\inert\gen}\|\freq\gen\|^2 \leq \|\freq\gen\|^2_{\inert\gen} \leq \lmax{\inert\gen}\|\freq\gen\|^2, \]
\[ 0 < \lmin{\balance}\|\diff\ctrl\|^2 \leq \|\diff\ctrl\|^2_{\balance} \leq \lmax{\balance}\|\diff\ctrl\|^2. \]
Furthermore, by Lemma~\ref{lem:activepower} in the Appendix,
\begin{multline*}
0 < \al{1}\|\diff\aedge\|^2 \leq \pot(\aedge) - \pot(\bar \aedge) - \nabla \pot(\bar \aedge)\T (\diff\aedge) \\\leq \al{2}\|\diff\aedge\|^2\!\!.
\end{multline*}
Therefore, if the cross-terms were absent, one would find 
\begin{align*}
\lowerbd &= \min \left( \tfrac12 \lmin{\inert\gen}, \tfrac12 \lmin\balance, \al{1} \right), \\
\upperbd &= \max \left( \tfrac12 \lmax{\inert\gen}, \tfrac12 \lmax\balance, \al{2} \right).
\end{align*}

Next, let us estimate the deviation caused by the cross-terms \eqref{e:lyap2}, for which we will use the following consequence of Young's inequality and the triangle inequality: for two vectors $a$, $b$,
\begin{multline*}
2|a\T b| = 2\left| \sum_i a_i b_i \right| \leq 2 \sum_i |a_i b_i| \\
   \leq \sum_i (a_i^2 + b_i^2) = \|a\|^2 + \|b\|^2.
\end{multline*}
Similarly,
\[ 2|a\T b| \geq -\|a\|^2 - \|b\|^2. \]

Consider $\lyap_1 := (\potdiff)\T \balance \inert \freq$.
Using the above,
\begin{multline*}
  -\|\balance(\potdiff)\|^2 - \|\inert \freq\|^2 \leq 2 \lyap_1 \\
  \leq \|\balance(\potdiff)\|^2 + \|\inert \freq\|^2
\end{multline*}
Then, using Lemma~\ref{lem:activepower}, and noting that $\inert \freq = \inert\gen \freq\gen$ by definition of \inert, we find
\begin{multline*}
  -\al{2} \lmax{\balance} \|\diff\aedge\|^2 - \lmax{\inert\gen}^2 \|\freq\gen\|^2 \leq 2 \lyap_1 \\
  \leq \al{2} \lmax{\balance} \|\diff\aedge\|^2 + \lmax{\inert\gen}^2 \|\freq\gen\|^2.
\end{multline*}

Next, we consider $\lyap_2 := (\diff{\ctrl})\T \ones \inert \freq$.
Using the above,
\begin{multline*}
  -\|\ones (\diff\ctrl)\|^2 - \|\inert \freq\|^2 \leq 2 \lyap_2 \\
  \leq \|\ones (\diff\ctrl)\|^2 + \|\inert \freq\|^2,
\end{multline*}
and, since $\lmax{\ones} = \cardin,$
\begin{multline*}
  -\cardin^2 \|\diff\ctrl\|^2 - \lmax{\inert\gen}^2 \|\freq\gen \|^2 \leq 2 \lyap_2 \\
  \leq \cardin^2 \|\diff\ctrl\|^2 + \lmax{\inert\gen}^2 \|\freq\gen\|^2.
\end{multline*}

As a result, the entire Lyapunov function is bounded as in the Lemma, with
\begin{align*}
\lowerbd &= \tfrac12 \min \big( \lmin{\inert\gen} 
                          - (\ep{1} + \ep{2}) \lmax{\inert\gen}^2, \\
         &\qquad        \lmin\balance - \ep{2} \cardin^2,
                        2 \al{1} - \ep{1}\al{2}\lmax{\balance} \big), \\
\upperbd &= \tfrac12 \max \big( \lmax{\inert\gen}
                          + (\ep{1} + \ep{2}) \lmax{\inert\gen}^2, \\
         &\qquad         \lmax\balance + \ep{2} \cardin^2,
                        2 \al{2} + \ep{1}\al{2}\lmax{\balance} \big).
\end{align*}
Here, \upperbd is trivially positive, while \lowerbd can be made positive by choosing \ep{1} and \ep{2} sufficiently small.
\end{proof*}

\begin{lemma} \label{lem:activepower}
Consider \aedge and $\pot(\aedge)$ as defined in Section~\ref{sec:setting},
and the Bregman distance $\lyap_\aedge := \pot(\aedge) - \pot(\bar\aedge) - \nabla\pot(\bar\aedge)\T(\diff\aedge)$.
The following properties hold for all $\aedge, \bar\aedge$ that satisfy $\incidence\T \aedge, \incidence\T \bar\aedge \in \rotorset$:
\begin{enumerate}
	\item \label{enumitem:lem-activepower-bound-potdiff}
	There exist positive scalars \al{1} and \al{2} such that
	\[ \al{1} \|\diff\aedge\|^2 \leq \|\potdiff\| \leq \al{2} \|\diff\aedge\|^2. \]
	\item \label{enumitem:lem-activepower-bound-bregman}
	There exist positive scalars \be{1} and \be{2} such that
	\[ \be{1} \|\diff\aedge\|^2 \leq \lyap_\aedge \leq \be{2} \|\diff\aedge\|^2. \]
	\item \label{enumitem:lem-activepower-bound-genonly}
	There exists a positive scalar $\gamma$ such that
	\[ \|\statevec\gen\|^2 \leq \|\statevec\|^2 \leq \gamma \|\statevec\gen\|^2. \]
\end{enumerate}
\end{lemma}

\begin{proof}
In the following, we denote by $L(\redge)$ the Laplacian matrix $\incidence\T \vedge [\cos(\eta)] \incidence$, for $\redge \in \R^\cardinedge$.

\emph{Proof of (\ref{enumitem:lem-activepower-bound-potdiff}).}
The vector \potdiff is defined as $\incidence \vedge (\sin(\incidence\T \aedge) - \sin(\incidence\T \bar\aedge))$.
Applying the mean value theorem componentwise to the difference vector $\sin(\incidence\T \aedge) - \sin(\incidence\T \bar\aedge)$ yields a vector $\tilde\aedge_i$ for each component as a function of \aedge and $\bar\aedge$, such that
\[ \sin(\incidence_i\T \aedge) - \sin(\incidence_i\T \bar\aedge) = \cos(\incidence_i\T \tilde\aedge_i) \incidence_i\T (\diff\aedge). \]
Stacking the result, and writing $\tilde\redge \in \R^\cardinedge$ such that $\tilde\redge_i := \incidence_i\T \tilde\aedge_i$, we arrive at 
\[ \sin(\incidence\T \aedge) - \sin(\incidence\T \bar\aedge) = [\cos \tilde\redge] \incidence\T (\diff\aedge). \]
Given that \aedge and $\bar\aedge$ satisfy the security constraint, each of the $\aedge_i \in \rotorset$, and therefore, $\tilde\redge \in \rotorset$.
By pre-multiplying by $\incidence\vedge$, we find $\potdiff = L(\tilde\redge) (\diff\aedge)$.
Given that $\tilde\redge \in \rotorset$, $L(\tilde\redge)$ is a Laplacian matrix, and therefore positive semi-definite with $\ker L(\tilde\redge) = \range\1$.
Since by definition, $\aedge, \bar\aedge \perp \1$,
\begin{multline*}
\lambda_2{L(\tilde\redge)}^2 \|\diff\aedge\|^2 \leq \|\potdiff\| \\\leq \lmax{L(\tilde\redge)}^2 \|\diff\aedge\|^2.
\end{multline*}
Remembering that $\tilde\redge$ depends on $\aedge$ and $\bar\aedge$, the result holds with
\[ \al{1} := \min_{\incidence\T \aedge, \incidence\T \bar\aedge \in \rotorset} \lambda_2(L(\tilde\redge))^2 \]
and
\[ \al{2} := \max_{\incidence\T \aedge, \incidence\T \bar\aedge \in \rotorset} \lmax{L(\tilde\redge)}^2. \]

\emph{Proof of (\ref{enumitem:lem-activepower-bound-bregman}).}
We claim that the Bregman distance $\pot(\aedge) - \pot(\bar\aedge) - \nabla\pot(\bar\aedge)\T (\diff\aedge)$ can be written as
\[ \lyap_\aedge = (\diff\aedge)\T L(\redge') (\diff\aedge) \]
for some $\redge' \in \rotorset$ that depends on \aedge and $\bar\aedge$.
To see this, we write $\lyap_\aedge$ as a function of $\redge := \incidence\T \aedge$, and likewise for $\bar\redge$. Then
\[ \lyap_\aedge = \tilde\pot(\redge) - \tilde\pot(\bar\redge) - \nabla\tilde\pot(\bar\redge)\T (\diff\redge), \]
setting $\tilde\pot(\redge) := -\1\T\vedge \cos\redge$, so $\nabla\tilde\pot(\redge) = \vedge\sin\redge$.
Since $\nabla^2\tilde\pot(\redge) = \vedge\cos\redge > 0$ for $\redge \in \rotorset$, and since \rotorset is closed, there exists a positive number \hesslev such that $\nabla^2 \tilde\pot(\redge)\ge \hesslev I$ for any $\redge \in \rotorset$.
This implies that $\tilde\pot(\redge)$ is a strongly convex function, and as a consequence, the Bregman distance is equal to 
\[ (\diff\redge)\T \nabla^2 \tilde\pot(\redge') (\diff\redge) \] 
for some $\tilde\redge$ whose elements are a convex combination of those of \redge and $\bar\redge$ \citep[Section 9.1.2, page 459]{boyd04convex}.
We then rewrite this in $\aedge$--coordinates to obtain the claim.

Since, once again, $\aedge \perp \1$, we have
\[ \lambda_2{L(\redge')} \|\diff\aedge\|^2 \leq \lyap_\aedge \leq \lmax{L(\redge')} \|\diff\aedge\|^2. \]
Therefore the result holds with
\[ \be{1} := \min_{\incidence\T \aedge, \incidence\T \bar\aedge \in \rotorset} \lambda_2(L(\redge')) \]
and
\[ \be{2} := \max_{\incidence\T \aedge, \incidence\T \bar\aedge \in \rotorset} \lmax{L(\redge')}. \]
\emph{Proof of (\ref{enumitem:lem-activepower-bound-genonly}).}
The first inequality follows immediately from the fact that $\statevec\gen$ is obtained by omitting the $\freq\load$ elements from \statevec.
The second one follows from the third equation in \eqref{e:incrementalswing}, which can be written as
\[ \damp\load \freq\load = - (\potdiff)\load + \diff{\ctrl\load}. \]
Hence, by Claim \ref{enumitem:lem-activepower-bound-potdiff} of this Lemma, we also have 
\begin{align*}
\|\damp\load \freq\load\| &= \|(\potdiff)\load - (\diff{\ctrl\load})\| \\
                       &\leq \|(\potdiff)\load\| + \|\diff{\ctrl\load}\| \\
                       &\leq \|(\potdiff)\| + \|\diff{\ctrl}\| \\
                       &\leq \al{2} \|\diff\aedge\| + \|\diff\ctrl\|.
\end{align*}
As a result, $\|\freq\load\| \leq \frac1{\lmax{\damp\load}} (\al{2} \|\diff\aedge\| + \|\diff\ctrl\|) \leq \frac1{\lmax{\damp\load}} (\al{2}+1) \|\statevec\gen\|$.
We conclude that
\begin{align*}
\|\statevec\|^2 &= \|\statevec\gen\|^2 + \|\freq\load\|^2 \\
                &\leq \left( 1 + \left(\frac{\al{2}+1}{\lmax{\damp\load}}\right)^2 \right) \|\statevec\gen\|^2,
\end{align*}
so $\gamma = 1 + \left( (\al{2}+1)/\lmax{\damp\load} \right)^2$.
\end{proof}

\begin{remark}
The bounds derived in the proof of this Lemma are general, but conservative. If the equilibrium value of $\bar\aedge$ is known, one can increase the potential tightness of $\al{i}, \be{i}$ by calculating the minima and maxima for this fixed value of $\bar\aedge$ and over \aedge with $\incidence\T \aedge \in \rotorset$.
\end{remark}

\begin{lemma} \label{lem:removecrossterms}
Given four appropriately sized matrices $a,$ $b,$ $c$ and $d$,
\[ M := \begin{bmatrix} a & b\T c \\ c\T b & d \end{bmatrix} \geq
        \begin{bmatrix} a - b\T b & 0 \\ 0 & d - c\T c \end{bmatrix} =: M'. \]
\end{lemma}
\begin{proof}
For any appropriately sized pair of vectors $u$ and $v$, let $x\T := \begin{bmatrix} u\T & v\T \end{bmatrix}$. Then
\[ x\T M x = u\T a u + v\T d v + 2 u\T b\T c v. \]
Note that
\begin{multline*}
0 \leq (b u + c v)\T (b u + c v) \\
  = u\T b\T b u + v\T c\T c v + 2 u\T b\T c v,
\end{multline*}
hence
\[ 2 u\T b\T c v \geq - u\T b\T b u - v\T c\T c v. \]
As a result,
\[ x\T M x \geq u\T (a - b\T b) u + v\T (d - c\T c) v, \]
so the claim follows.
\end{proof}

\futureconsideration{To shorten the paper and emphasize the theoretical aspects many of the references on power nets and DoS can be omitted.}

\bibliography{bibliography}{}
\bibliographystyle{./IEEEtranSN}

\end{document}